\documentclass[a4paper,12pt,draft]{amsart}
\usepackage{amsmath, amssymb, amsthm, url}
\usepackage{amscd}
\usepackage{url}
\usepackage{hyperref}
\usepackage{cleveref}
\usepackage{braket}

\newcommand{\ZZ}{\mathbb{Z}}

\newcommand{\KK}{\mathbb{K}}

\newcommand{\der}{\partial}

\newcommand{\numof}[1]{|#1|}

\newcommand{\chara}{\operatorname{char}}
\newcommand{\Ann}{\operatorname{Ann}}

\newcommand{\MMM}{\mathcal{M}}

\newcommand{\lvrace}{{\boldsymbol\{}}
\newcommand{\rvrace}{{\boldsymbol\}}}

\newcommand{\ed}[2]{\lvrace #1,#2 \rvrace}

\theoremstyle{plain}   

\newtheorem{thm}{Theorem}[section]
\newtheorem{theorem}[thm]{Theorem}

\newtheorem{lemma}[thm]{Lemma}
\newtheorem{cor}[thm]{Corollary}

\theoremstyle{remark}  

\newtheorem{example}[thm]{Example}

\theoremstyle{definition}  

\crefname{section}{Section}{Sections}

\crefname{thm}{Theorem}{Theorems}
\Crefname{thm}{Theorem}{Theorems}
\crefname{theorem}{Theorem}{Theorems}
\Crefname{theorem}{Theorem}{Theorems}
\crefname{lemma}{Lemma}{Lemmas}
\Crefname{lemma}{Lemma}{Lemmas}
\crefname{corollary}{Corollary}{Corollaries}
\Crefname{corollary}{Corollary}{Corollaries}
\crefname{cor}{Corollary}{Corollaries}
\Crefname{cor}{Corollary}{Corollaries}
\crefname{proposition}{Proposition}{Propositions}
\Crefname{proposition}{Proposition}{Propositions}
\crefname{prop}{Proposition}{Propositions}
\Crefname{prop}{Proposition}{Propositions}
\crefname{remark}{Remark}{Remarks}
\Crefname{remark}{Remark}{Remarks}
\crefname{rem}{Remark}{Remarks}
\Crefname{rem}{Remark}{Remarks}
\crefname{example}{Example}{Examples}
\Crefname{Example}{Example}{Examples}
\crefname{exercise}{Exercise}{Exercises}
\Crefname{exercise}{Exercise}{Exercises}
\crefname{definition}{Definition}{Definitions}
\Crefname{definition}{Definition}{Definitions}
\crefname{dfn}{Definition}{Definitions}
\Crefname{dfn}{Definition}{Definitions}
\crefname{algorithm}{Algorithm}{Algorithms}
\Crefname{algorithm}{Algorithm}{Algorithms}
\crefname{question}{Question}{Questions}
\Crefname{question}{Question}{Questions}
\crefname{problem}{Problem}{Problems}
\Crefname{problem}{Problem}{Problems}
\crefname{notation}{Notation}{Notations}
\Crefname{notation}{Notation}{Notations}
\crefname{conjecture}{Conjecture}{Conjectures}
\Crefname{conjecture}{Conjecture}{Conjectures}
\crefname{conj}{Conjecture}{Conjectures}
\Crefname{conj}{Conjecture}{Conjectures}
\crefname{condition}{Condition}{Conditions}
\Crefname{condition}{Condition}{Conditions}

\allowdisplaybreaks[3]

\newcommand{\nin}{\cup\kern-0.352em\rule{0.4pt}{1.4ex}\kern-0,4pt\rule{0.352em}{0ex}}

\begin{document}


\def\Time-stamp: %
<#1 #2 #3>{\def\lastupdatetimestamp{#1 #2}}
\Time-stamp: <2022-03-01 09:41:09 nous>

\author[Y. Numata]{Yasuhide NUMATA}
\address[Y. Numata]{Department of Mathematics, Hokkaido University, Sapporo, Japan.}
\thanks{The author was partially supported by JSPS KAKENHI Grant Number JP18K03206.}
\email{nu@math.sci.hokudai.ac.jp}
\title[An algebra defined by matchings]{The Lefschetz property for an algebra defined by matchings}
\begin{abstract}
In this article, we consider the weighted generating function of
matchings consisting of $k$ edges in the complete graph.
We define an Artinian Gorenstein algebra  as
the quotient ring of a polynomial ring
by the annihilator of the generating function.
We show the strong Lefschetz property of the algebra.
\end{abstract}
\maketitle
\section{Introduction}
The Lefschetz property of Artinian algebras is
an algebraic abstraction of the property of the cohomology rings of  
compact K\"aler manifolds
known as the Hard Lefschetz Theorem,
see also \cite{MR3112920, MR3161738}.
For a graded algebra $A=\bigoplus_{d=0}^s A^{(d)}$ with $A^{(d)}\neq 0$,
we say that 
$A$ has the \emph{strong Lefschetz property} with a Lefschetz element $l\in A^{(1)}$
if the  multiplication map
\begin{align*}
\times l^{s-2d}\colon A^{(d)} &\to A^{(s-d)}\\
       f&\mapsto fl^{s-2d}
\end{align*}
 is bijective for each $d$ with $2d\leq s$.
For the strong Lefschetz property of Artinian Gorenstein algebras,
a criterion using higher Hessian matrices is known \cite{W2000,MR2594646},
see also \cref{lem:hessiancriterion}.
The strong Lefschetz property is studied
from the viewpoint of theory of not only algebra but also combinatorics.
Recently, some authors study
Artinian Gorenstein algebras defined by combinatorial polynomials and
calculate the determinants of higher Hessian matrices of the polynomials
to show the strong Lefschetz property of the algebras.
For example, 
the basis generating functions for matroids are studied in \cite{MR3566530,MR4308087,MR4348920,MR4234203,MR4223331},
and 
the weighted generating functions of facets of Platonic solids are studied in \cite{2011.13732}.
In this article, we consider the generating function of matchings,
i.e., sets of edges which does not share vertices, consisting of $k$ edges in the complete graphs.
We show the strong Lefschetz property of the Artinian Gorenstein algebras defined by the weighted generating function.

This article organized as follows:
In \cref{sec:ring}, we recall definitions and know facts of Lefschetz property of Artinian Gorenstein algebras.
In \cref{sec:matching}, we define the generating function of matchings, and calculate higher Hessian matrices
to show the strong Lefschetz property of the  Artinian Gorenstein algebra
defined by the generating function of matchings.

\section{Lefschetz properties and higher Hessians}
\label{sec:ring}
Let $\KK$ be a field with $\chara(\KK)=0$.
Consider the polynomial ring $P=\KK[x_i|i\in U]$
in variables indexed by a finite set $U$ with $\deg(x_i)=1$.
For a homogeneous polynomial $\Phi\in P$,
we define the annihilator $\Ann(\Phi)$ of $\Phi$ by
\begin{align*}
\Ann(\Phi)=\Set{f\in P| f(\der) \Phi = 0},
\end{align*}
where $f(\der)$ is a partial derivative operator
obtained from $f$ by substituting $\der_i=\frac{\der}{\der x_i}$ to $x_i$ for 
each  $i\in U$.
Since the annihilator $\Ann(\Phi)$ is a homogeneous ideal,
the quotient ring $A=P/\Ann(\Phi)$ is a graded ring.
Let $A^{(d)}$ be the homogeneous component of degree $d$.
Assume that $\deg(\Phi)=s$.
Then we have $A^{(s)}\simeq A^{(0)} \simeq \KK$ and $A^{(d)}=0$ for $d>s$.
We call the quotient ring $A=P/\Ann(\Phi)$ for some homogeneous polynomial
an \emph{Artinian Gorenstein algebra} with standard grading.
It is know that
the bilinear map, called the \emph{Poincar\'e duality}, defined by
\begin{align*}
A^{(d)}\times A^{(s-d)} &\to \KK\\
(f,g) & \mapsto f(\der)g(\der) \Phi
\end{align*}
is non-degenerate.
Sometimes, the algebra $A$ is also called a \emph{Poincar\'e duality algebra}.
For an Artinian Gorenstein algebra with standard grading,
the following criterion for the bijectivity of the multiplication map $\times l^{s-2d}\colon A^{(d)} \to A^{(s-d)}$ is known.
\begin{lemma}[\cite{W2000,MR2594646}]
\label{lem:hessiancriterion}
Let $\Phi\in P$ be a homogeneous polynomial of degree $s$.
Let $A^{(d)}$ be the set  of homogeneous elements in the quotient ring $A=P/\Ann(\Phi)$ of degree $d$.
Assume that the image of $B\subset P$ is a basis for the $\KK$-vector space $A^{(d)}$.
Let $H$ be the matrix
$\left(F(\der)F'(\der)\Phi\right)_{F, F' \in B}$.
For $l=\sum_{j \in U} a_j x_j \in A^{(1)}$,
the multiplication map
$\times l^{s-2d}\colon A^{(d)}\to A^{(s-d)}$ is bijective
if and only if
 $\det(H)|_{\forall i,\ x_i=a_i} \neq 0$.
\end{lemma}
We call the matrix $H$ a $d$-th \emph{Hessian matrix} with respect to the basis $B$.
By \cref{lem:hessiancriterion},
if the determinant of a $d$-th Hessian matrix is a nonzero polynomial,
then the set of elements $l\in A^{(1)}$ which induces
the bijective multiplication map
$\times l^{s-2d}\colon A^{(d)}\to A^{(s-d)}$
is an open dense subset in $A^{(1)}$.
Hence
$A$ has the strong Lefschetz property if and only if
the $d$-th  Hessian matrix are nonzero polynomials
for all $d$ with $2d\leq s$.

In \cite{MR3566530},
the Lefschetz property of Artinian Gorenstein algebras defined by the basis generating functions of matroids
are studied.
In the extended abstract \cite{MR2985388} of \cite{MR3566530},
the elementary symmetric polynomial,
which is the basis generating function of a uniform matroid,
is considered as an example.
\begin{example}
\label{lem:elementary}
For a set $U$,
we define $\binom{U}{k}$ to be
the set of subsets of $U$ consisting of $k$ elements.
We define the $k$-th \emph{elementary symmetric polynomial} $e_k(x_i|i\in U)$ by
\begin{align*}
e_k(x_i|i\in U) = \sum_{S \in \binom{U}{k}} x^S,
\end{align*}
where $x^{S}=\prod_{i\in S} x_i$. 
By \cite[Theorem 4.3]{MR3566530},
$A=P/\Ann(e_k(x_i|i\in U))=\bigoplus_{i=0}^{k}A^{(i)}$
has the strong Lefschetz property.
Moreover, if $2d\leq k$, then
the set of images of all monomials of degree $d$ is a basis for the $\KK$-vector space $A^{(d)}$.
Hence the Hilbert series $h=(h_0,\ldots,h_k)$ is 
\begin{align*}
\left(\binom{k}{0},\binom{k}{1},\ldots,\binom{k}{k} \right)
\end{align*}
where $h_d$ is the dimension of the  $\KK$-vector space $A^{(d)}$.
The following matrix is a $d$-th Hessian matrix
\begin{align*}
\left(
\der^{S}\der^{S'} e_k(x_i|i\in U)
\right)_{S,S'\in \binom{U}{d}},
\end{align*}
where $\der^S=\prod_{i\in S}\der_i$.
Since the algebra $A$ has the strong Lefschetz property,
 the determinant 
is a nonzero polynomial.
If  $k=2d$, then
we have
\begin{align*}
\der^{S}\der^{S'} e_{2d}(x_i|i\in U) =
\begin{cases}
1& (S\cap S' = \emptyset)\\
0& (S\cap S' \neq \emptyset)
\end{cases}
\end{align*}
for $S,S' \in\binom{U}{d}$.
In this case, the determinant of the $d$-th Hessian matrix,
which is also a representation matrix of
the Poincar\'e duality $A^{(d)}\times A^{(d)} \to \KK$,
is a nonzero constant.
\end{example}

\section{Weighted generating functions of matchings}
\label{sec:matching}
In this article, we consider only simple nondirected graphs.
Hence an edge of a graph with vertex set $U$
is regard as an element in $\binom{U}{2}$.
Moreover we assume that the vertex set $U$ is a nonempty finite subset of the set $\ZZ$ of integers.
We define the \emph{complete graph} $K_U$ to be the graph with the vertex set $U$ and the edge set $\binom{U}{2}$.
For a graph, a set of edges is called a \emph{matching} in the graph
if any distinct edges in the set does not share an vertex.
In other words, a subset $M$ of the edge set of a graph is a matching
if $e,e'\in M$ and $e\neq e'$ implies $e\cap e' = \emptyset$.
We define  $\MMM(U)$ to be the set of matchings in $K_U$.
We also define the set  $\MMM(U,k)$ to be
the set $\Set{M\in \MMM(U)| \numof{M}=k}$ consisting of
matchings with $k$ edges.
For a finite set $U$ with $\numof{U}=u$,
we have $\numof{\MMM(U,k)}=\binom{u}{2k}(2k-1)!!$.
For a set $V\subset U$ with $V=\Set{v_1 < \ldots < v_{2d}}$, 
we define $M(V)$ to be the matching $\Set{\ed{v_{2i-1}}{v_{2i}}| i=1,\ldots,d} \in \MMM(V,d)\subset \MMM(U,d)$.

We define $P_U$ to be the polynomial ring $\KK[x_e|e\in \binom{U}{2}]$
in variables corresponding to edges in the complete graph $K_U$.
We define the weighted generating function $\Phi_{U,k}$ of $\MMM(U,k)$ by
\begin{align*}
\Phi_{U,k} = \sum_{M\in \MMM(U,k)} x^M \in P_U.
\end{align*}
By definition
 $\Phi_{U,k}$ is a square-free symmetric homogeneous polynomial of degree $k$.
We define $A_{\Phi_{U,k}}$ to be the quotient ring $P_U/\Ann(\Phi_{U,k})$.
Our main result is the following:
\begin{theorem}[Main result]
\label{thm:main}
For a finite set $U$,
the Artinian Gorenstein algebra $A_{\Phi_{U,k}}$ has the strong Lefschetz property with
a Lefschetz element $l=\sum_{e\in \binom{U}{2}}x_e$.
\end{theorem}

To show the main theorem,
we show some lemma:
\begin{lemma}
\label{lem:dualpoly}
Let $V$ be a subset of $U$ with $\numof{V}=2d$.
For a matching $M\in\MMM(V,d)$, $\der^{M} \Phi_{U,k} = \Phi_{U\setminus V,k-d}$.
\end{lemma}
\begin{proof}
Since the map
\begin{align*}
\MMM(U\setminus V,k-d) &\to \Set{\tilde M\in \MMM(U,k) | M\subset \tilde M }\\ 
M'&\mapsto  M' \cup M
\end{align*}
is bijective, we obtain the equation by direct calculation.
\end{proof}

\begin{lemma}
\label{lem:lingen}
Let 
\begin{align*}
G_0 &= \Set{x_e^2|e\in \binom{U}{2}},\\
G_1 &= \Set{x^S| \text{$S\subset \binom{U}{2}$ is not a matching in $K_{U}$}},\\
G_2 &= \Set{x^M-x^{M'}| \text{$M,M' \in \MMM(V,d)$ for some $V\in \binom{U}{2d}$}}.
\end{align*}
The annihilator $\Ann(\Phi_{U,k})$ contains  $G=G_1\cup G_1 \cup G_2$.
\end{lemma}
\begin{proof}
Since $\Phi_{U,k}$ is a square-free polynomial,
it is obvious that the annihilator $\Ann(\Phi_{U,k})$ contains  $G_0$.
Since a subset of matching is also matching,
if $S$ is not a matching, then there does not exists a matching containing $S$.
Hence the annihilator $\Ann(\Phi_{U,k})$ contains  $G_1$.
For $M,M' \in \MMM(V,d)$,
it follows from \cref{lem:dualpoly} that 
\begin{align*}
(\der^{M} - \der^{M'}) \Phi_{U,k} = \Phi_{U\setminus V,k-d}- \Phi_{U\setminus V,k-d} = 0
\end{align*}
for  $V\in \binom{U}{2d}$ and $M,M' \in \MMM(V,d)$.
Hence the annihilator $\Ann(\Phi_{U,k})$ contains  $G_2$.
\end{proof}

By \cref{lem:lingen},
if $2d\leq k$, then
the image of $\Set{ x^{M(V)} | V\in \binom{U}{2d}}$
generates the  $\KK$-vector space $A_{V,k}^{(d)}$.
To show that the set is linearly independent,
we calculate the determinant of the following matrix:
\begin{align*}
H_{U,k,d}=
\left(
\der^{M(V)}\der^{M(V')} \Phi_{U,k}
\right)_{V,V' \in \binom{U}{2d}}.
\end{align*}

\begin{lemma}
\label{lem:hessianentry}
Let $V,V' \in \binom{U}{2d}$ with $2d\leq k$ and $\numof{U}=u$.
For $M\in \MMM(V,d)$ and $M'\in \MMM(V',d)$,
\begin{align*}
\der^{M}\der^{M'} \Phi_{U,k} =
\begin{cases}
\Phi_{U\setminus (V\cup V'),k-2d} &(V\cap V' = \emptyset) \\
0 &(V\cap V' \neq \emptyset).
\end{cases}
\end{align*}
Hence
\begin{align*}
\der^{M'}\der^{M} \Phi_{U,k}|_{\forall e,\ x_e=1}
&=
\begin{cases}
\binom{u-2d}{2k-4d}(2k-4d-1)!!
 &(V\cap V' = \emptyset) \\
0 &(V\cap V' \neq \emptyset)
\end{cases}\\
&=
\binom{u-2d}{2k-4d}(2k-4d-1)!!
\der^{V'}\der^{V} e_{2d}(x_i|i \in U).
\end{align*}
\end{lemma}

\begin{proof}
By \cref{lem:dualpoly},
we have
\begin{align*}
\der^{M} \Phi_{U,k} = \Phi_{U\setminus V,k-d}.
\end{align*}
If $V\cap V' \neq \emptyset$,
then
$U\setminus V$ does not contain $V'$.
Hence any matching $\check M \in \MMM(U\setminus V)$
does not contain $M' \in \MMM(V', 2k)$.
In this case, $\der^{M}\der^{M'} \Phi_{U,k} = 0$
If $V\cap V' = \emptyset$,
then 
$\der^{M'} \Phi_{U\setminus V,k-d} = \Phi_{U\setminus (V\cup V'),k-2d}$
by \cref{lem:dualpoly}.
Since 
 $\numof{\MMM(S,k)}=\binom{\numof{S}}{2k}(2k-1)!!$
 for a finite set $S$,
we have $\der^{M'}\der^{M} \Phi_{U,k}|_{\forall e,\ x_e=1} = \binom{u-2d}{2k-4d}(2k-4d-1)!!$ 
in this case.
Since
\begin{align*}
\der^{S}\der^{S'} e_{2d}(x_i|i\in U) =
\begin{cases}
1& (S\cap S' = \emptyset)\\
0& (S\cap S' \neq \emptyset)
\end{cases}
\end{align*}
for $S,S' \in\binom{U}{d}$,
we have
\begin{align*}
\der^{M'}\der^{M} \Phi_{U,k}|_{\forall e,\ x_e=1}
&=
\binom{u-2d}{2k-4d}(2k-4d-1)!!
\der^{V'}\der^{V} e_{2d}(x_i|i \in U)
\end{align*}
for $M\in \MMM(V,d)$ and $M'\in \MMM(V',d)$.
\end{proof}

By
\cref{lem:hessianentry},
we have
\begin{align*}
\det(H_{U,k,d})|_{\forall e,\ x_e=1}
=
c
\det\left(
\der^{V'}\der^{V} e_{2d}(x_i|i \in U)
\right)_{V,V' \in \binom{U}{2d}},
\end{align*}
where
$c=(\binom{u-2d}{2k-4d}(2k-4d-1)!!)^{\binom{u}{2d}}$
and $u=\numof{U}$.
Since the determinant
$\det\left(\der^{V'}\der^{V} e_{2d}(x_i|i \in U)\right)_{V,V' \in \binom{U}{2d}}$
is not zero as in \cref{lem:elementary},
the determinant 
$\det(H_{U,d,k})|_{\forall e,\ x_e=1}$ is a nonzero constant.
Since $\det(H_{U,d,k})$ is a nonzero polynomial,
the image of  $\Set{ x^{M(V)} | V\in \binom{U}{2d}}$
is linearly independent.
Hence we have the following corollary.
\begin{cor}
The Hilbert series of $A_{U,k}$ with $2k\leq \numof{U}$ is
\begin{align*}
\left(
\binom{2k}{0},\binom{2k}{2},\binom{2k}{4},\ldots,\binom{2k}{2k}
\right)
\end{align*}
\end{cor}

Since 
the image of the set $\Set{ x^{M(V)} | V\in \binom{U}{2d}}$
is a basis for the  $\KK$-vector space $A_{V,k}^{(d)}$,
the matrix $H_{U,d,k}$ is a $d$-th Hessian matrix for $A_{V,k}^{(d)}$
for $d$ with $2d\leq k$.
By 
\cref{lem:hessiancriterion},
we have the strong Lefschetz property for $A_{V,k}$,
which is our main result \cref{thm:main}.

\bibliographystyle{amsplain-url}
\bibliography{by-mr,by-arxiv,by-nameyear}

\end{document}